\documentclass[12pt,reqno]{amsart}
\usepackage[a4paper,margin=2.5cm,top=2.5cm,bottom=2.5cm,centering,headheight=3ex,headsep=4ex,vcentering]{geometry}

\usepackage{amssymb,amsfonts,amsmath,amsthm,booktabs}
\usepackage{orcidlink,tikz}

\usepackage[cal=cm,scr=euler]{mathalfa}
\usepackage{microtype}
\usepackage{mlmodern}

\theoremstyle{plain}
\newtheorem{theorem}{Theorem}[section]

\newtheorem{lemma}[theorem]{Lemma}

\theoremstyle{definition}

\newtheorem{remark}[theorem]{Remark}
\theoremstyle{plain}

\theoremstyle{definition}

\newcommand{\black}{\operatorname{black}}
\newcommand{\white}{\operatorname{white}}

\newcommand{\col}{\operatorname{col}}
\newcommand{\Nar}{N}

\newcommand{\Raney}{R}
\newcommand{\Z}{\mathbb{Z}}

\newcommand{\ceil}[1]{\left\lceil #1 \right\rceil}

\title[Dyck paths on colored lattices]{Dyck paths on colored lattices}

\author[M. P. Saikia]{Manjil P. Saikia\,\orcidlink{0000-0002-2997-6731}}
\address[M. P. Saikia]{Mathematical and Physical Sciences division, School of Arts \& Sciences, Ahmedabad University, Navrangpura, Ahmedabad 380009, Gujarat, India}
\email{manjil.saikia@ahduni.edu.in}

\date{Version of July 7, 2026}
\thanks{Claude Opus 4.8 and Claude Sonnet 5 were used to edit the note, and make notations more efficient. The author takes responsibility for all of the mathematical content.}



\begin{document}

\begin{abstract}
Fried recently enumerated Dyck paths having equally many black and white cells
below them, for the chessboard coloring (Narayana numbers) and the
column-alternating coloring (Fuss--Catalan numbers). We prove a
generalization here: for the coloring of columns modulo any $c\ge2$, the
number of Dyck paths of semilength $n$ whose $c$ residue classes carry equal
weight is the Raney number $\Raney_{c+1,r}(m)$, where $n=cm+r-1$.
\end{abstract}
\maketitle
\section{Introduction}\label{sec:intro}

A \emph{Dyck path} of semilength $n$ is a lattice path from $(0,0)$ to $(n,n)$
with unit east (E) and north (N) steps that never rises above the diagonal
$y=x$ (see Figure \ref{f11} for an example). Following Fried~\cite{Fried}, for $i\in\{1,\dots,n\}$ let $p_i$ be the
number of cells in the $i$th column below the path, equivalently the number of
N-steps taken before the $i$th E-step. Then it is easy to see that
\[
0=p_1\le p_2\le\cdots\le p_n \quad \text{and} \qquad p_i\le i-1.
\]
These column-height vectors are in bijection with Dyck paths, which are
counted by the famous Catalan number $C_n=\dfrac1{n+1}\dbinom{2n}{n}$ (see \cite[A000108]{OEIS} for several interpretations and other connections). 

\begin{figure}[htb!]
\centering
\scalebox{0.4}{%
\begin{tikzpicture}
\tikzset{DyckPath/.style={black,line width=4pt,line cap=round}}
\newcommand{\GridThree}{%
  \draw[very thin] (0,0) grid (3,3);}
\begin{scope}[shift={(0.0,0.0)}]    
  \GridThree
  \draw[DyckPath] (0,0)--(1,0)--(2,0)--(3,0)--(3,1)--(3,2)--(3,3);
\end{scope}
\begin{scope}[shift={(4.2,0.0)}]    
  \GridThree
  \draw[DyckPath] (0,0)--(1,0)--(2,0)--(2,1)--(3,1)--(3,2)--(3,3);
\end{scope}
\begin{scope}[shift={(8.4,0.0)}]     
  \GridThree
  \draw[DyckPath] (0,0)--(1,0)--(2,0)--(2,1)--(2,2)--(3,2)--(3,3);
\end{scope}
\begin{scope}[shift={(12.6,0.0)}]    
  \GridThree
  \draw[DyckPath] (0,0)--(1,0)--(1,1)--(2,1)--(3,1)--(3,2)--(3,3);
\end{scope}
\begin{scope}[shift={(16.8,0.0)}]    
  \GridThree
  \draw[DyckPath] (0,0)--(1,0)--(1,1)--(2,1)--(2,2)--(3,2)--(3,3);
\end{scope}
\end{tikzpicture}}
\caption{The five Dyck paths of semilength \(3\).}\label{f11}
\end{figure}
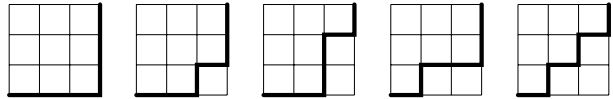

We now color the cells of the integer grid black/white in one of two ways (see Figure \ref{ff22}):
\begin{itemize}
\item \textbf{Chessboard:} cell $(i,j)$ is black iff $i+j$ is even.
\item \textbf{Column-alternating:} column $j$ is black iff $j$ is odd.
\end{itemize}
Let $\black(p)$ and $\white(p)$ be the numbers of black and white cells below
$p$, and set $D(p)=\black(p)-\white(p)$. Call $p$ \emph{balanced} if $D(p)=0$.
Fried~\cite{Fried} proved that the number of chessboard-balanced Dyck paths of
semilength $n$ is the Narayana number $\Nar(n,\ceil{n/2})$, where
$\Nar(n,k)=\dfrac1n\dbinom nk\dbinom n{k-1}$, and that the number of
column-alternating-balanced Dyck paths is
$\dfrac1{2m+1}\dbinom{3m}{m}$ for $n=2m$ and $\dfrac1{m+1}\dbinom{3m+1}{m}$ for
$n=2m+1$, that is, the Fuss--Catalan sequences
\href{https://oeis.org/A001764}{A001764} and
\href{https://oeis.org/A006013}{A006013}.

\begin{figure}[htb!]
\centering
\scalebox{0.4}{%
\begin{tikzpicture}
\newcommand{\ChessFour}{
  \foreach \i in {0,...,3}{
    \foreach \j in {0,...,3}{\pgfmathtruncatemacro\p{mod(\i+\j,2)}
      \ifnum\p=0 \fill[black!20] (\i,\j) rectangle ++(1,1);\fi}}
  \draw[very thin] (0,0) grid (4,4);
}
\newcommand{\ColsFour}{
  \foreach \i in {0,...,3}{
    \ifnum\i=0 \foreach \j in {0,...,3}{\fill[black!20] (\i,\j) rectangle ++(1,1);} \fi
    \ifnum\i=2 \foreach \j in {0,...,3}{\fill[black!20] (\i,\j) rectangle ++(1,1);} \fi}
  \draw[very thin] (0,0) grid (4,4);
}
\begin{scope}[shift={(0,0)}]
  \ChessFour
  \node[font=\sffamily\Huge] at (2,-1.4) {(a) Chessboard coloring.};
\end{scope}
\begin{scope}[shift={(12.0,0)}] 
  \ColsFour
  \node[font=\sffamily\Huge] at (2,-1.4) {(b) Column-alternating coloring.};
\end{scope}
\end{tikzpicture}}
\caption{The two colorings.}\label{ff22}
\end{figure}
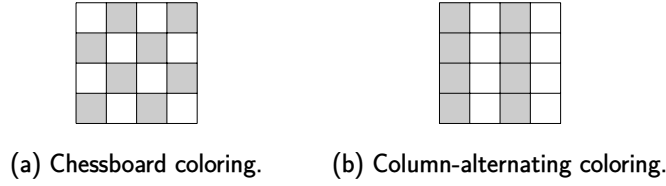

Our main result generalizes the column-alternating theorem to colorings
of any period. We formally state Fried's results first.
\begin{theorem}\cite[Theorem 1]{Fried}\label{thm:fried1}
For the chessboard coloring,
$D(p)=\sum_{i=1}^n(-1)^{i-1}\mathbf 1_{p_i\text{ odd}}$. Moreover, if $k$
denotes the number of N-steps of $p$ occurring at odd positions, then
$D(p)=\ceil{n/2}-k$, and consequently the number of chessboard-balanced Dyck
paths of semilength $n$ is $\Nar(n,\ceil{n/2})$. Here $\mathbf 1$ is the indicator function.
\end{theorem}

\begin{theorem}\cite[Theorem 2]{Fried}\label{thm:fried2}
For the column-alternating coloring, $D(p)=\sum_{i=1}^n(-1)^{i-1}p_i$. The number of black-white balanced Dyck paths of semilength $n$ is given by
\[
\begin{cases}
\dfrac{1}{2m+1}\dbinom{3m}{m},&\textnormal{if $n=2m$},\\
\dfrac{1}{m+1}\dbinom{3m+1}{m},&\textnormal{if $n=2m+1$}.
\end{cases}
\]
\end{theorem}

We color column $j$ by $j\bmod c$ and call $p$ \emph{$c$-balanced} if all $c$
residue classes carry equally many cells below $p$. Throughout this note we
write $n=cm+s$ with $s=n\bmod c$ and set $r=s+1\in\{1,\dots,c\}$. Since the number of cells
below $p$ in column $i$ is $p_i$, the number of cells of color $j$ below $p$
is the \emph{class sum}
\[
S_j \;=\; \sum_{\substack{1\le i\le n\\ i\equiv j \pmod{c}}} p_i,
\qquad j\in\Z/c\Z,
\]
and we call $p$ \emph{$c$-balanced} if $S_0=S_1=\cdots=S_{c-1}$.

\begin{theorem}[Generalization of Theorem~\ref{thm:fried2}]\label{thm:raney}
The number of $c$-balanced Dyck paths of semilength $n$ is the Raney number\footnote{Also called the two-parameter Fuss-Catalan number.}
\[
\Raney_{c+1,\,r}(m)\;=\;\dfrac{r}{(c+1)m+r}\dbinom{(c+1)m+r}{m}.
\]
\end{theorem}

The $c=2$ case is Theorem~\ref{thm:fried2}: $r=1$ ($n$ even) gives
\href{https://oeis.org/A001764}{A001764} and $r=2$ ($n$ odd) gives
\href{https://oeis.org/A006013}{A006013}. For $c=3$, we have
\[
c=3:\quad
\underbrace{\tfrac{1}{3m+1}\tbinom{4m}{m}}_{n\equiv0 \pmod 3},\quad
\underbrace{\tfrac{1}{2m+1}\tbinom{4m+2}{m}}_{n\equiv1 \pmod 3},\quad
\underbrace{\tfrac{3}{4m+3}\tbinom{4m+3}{m}}_{n\equiv2 \pmod 3},
\]
the first of which is \href{https://oeis.org/A002293}{A002293}, the second is \href{https://oeis.org/A069271}{A069271}, while the third is \href{https://oeis.org/A006632}{A006632}. 

We give a proof of this result in the next section.

\section{Proof of Theorem \ref{thm:raney}}\label{sec:raney}

We number the columns from left to right, starting with $1$ and going till $n$. We partition the columns into sets defined by $B_0=\{1,\dots,s\}$ and, for $t=1,\dots,m$,
$B_t=\{s+(t-1)c+1,\dots,s+tc\}$, and write $\col(t,a)=s+(t-1)c+a$
for the $a$th column of $B_t$, $1\le a\le c$. We call the $B_i$'s \textit{blocks}. Also, let $u_i=p_{i+1}-p_i\ge0$ for the gap sequence.

\begin{lemma}\label{lem:structure}
A Dyck path $p$ is $c$-balanced if and only if $p_i=0$ for all $i\in B_0$ and
$p$ is constant on each block $B_t$, $1\le t\le m$. Equivalently, the gaps
satisfy $u_k=0$ for all $k\not\equiv s\pmod c$.
\end{lemma}

\begin{proof}
The reverse direction is pretty straightforward, we only prove the forward direction.

For $a\in\{1,\dots,c\}$ set
$T_a=\sum_{t=1}^m p_{\col(t,a)}$. Since
$\col(t,a)+1=\col(t,a+1)$ for $a<c$, we have
\begin{equation}\label{eq:Tmono}
T_{a+1}-T_a=\sum_{t=1}^m(p_{\col(t,a+1)}-p_{\col(a,t)})=\sum_{t=1}^m u_{\col(t,a)}\;\ge\;0,
\end{equation}
so $T_1\le T_2\le\cdots\le T_c$, with equality at a step if and only if every
gap at position $a$ vanishes.

We next record how the residue classes decompose into blocks. Each block
$B_t$, $1\le t\le m$, consists of $c$ consecutive columns, so it contains
exactly one column from each residue class modulo $c$: its $a$-th column
$\col(t,a)=s+(t-1)c+a$ has residue
\[
\rho(a)\;\equiv\; s+a \pmod c,
\]
independently of $t$, and $a\mapsto\rho(a)$ is a bijection from
$\{1,\dots,c\}$ onto $\Z/c\Z$. Hence the class sum $S_{\rho(a)}$ receives the
contribution $p_{\col(t,a)}$ from every block, for a total of $T_a$, plus
whatever the initial segment $B_0$ contributes.

The columns of $B_0=\{1,\dots,s\}$ have residues $1,\dots,s$. To see which positions $a$
these residues correspond to, we read off $\rho(a)\equiv s+a$ as $a$ increases:
\begin{itemize}
\item for $a=1,\dots,c-s$, the values $s+a$ run over $s+1,\dots,c$, that is,
over the residues $s+1,\dots,c-1,0$, none of which meets $B_0$, and
\item for $a=c-s+1,\dots,c$, the values $s+a$ run over $c+1,\dots,c+s$,
which reduce to the residues $1,\dots,s$; explicitly
$\rho(a)=a-(c-s)$, and the unique column of $B_0$ in this class is the
column with that particular index, contributing its height $p_{\,a-(c-s)}$.
\end{itemize}
Therefore
\[
S_{\rho(a)}=T_a+\varepsilon_a,\qquad
\varepsilon_a=
\begin{cases}
0, & 1\le a\le c-s,\\[2pt]
p_{\,a-(c-s)}, & c-s+1\le a\le c.
\end{cases}
\]

Assume now that $p$ is balanced, so $T_a+\varepsilon_a$ is constant in $a$.
For $a=1,\dots,c-s+1$ we have $\varepsilon_a=0$, so
$T_1=\cdots=T_{c-s+1}$ and, by \eqref{eq:Tmono}, every gap
$u_{\col(t,a)}$ with $a\le c-s$ vanishes. For
$a=c-s+1,\dots,c-1$, we have
\[
\underbrace{T_{a+1}-T_a}_{\ge\,0\text{ by \eqref{eq:Tmono}}}
=\varepsilon_a-\varepsilon_{a+1}
=p_{\,a-(c-s)}-p_{\,a-(c-s)+1}
=\underbrace{-\,u_{\,a-(c-s)}}_{\le\,0},
\]
so both sides must be equal to $0$. Thus, $T_1=T_2=\cdots=T_c$, hence all block gaps vanish inside the blocks, which is what we want. Also, we have $u_1=u_2=\cdots=u_{s-1}=0$, which gives us $p_1=0$ and then successively $p_2=p_3=\cdots =p_s=0$.

This proves our lemma.
\end{proof}

\begin{remark}
  For $m=0$ (that is, $n<c$) the empty residue classes force all class sums to
vanish, so $p\equiv0$ is the unique balanced path, matching
$\Raney_{c+1,r}(0)=1$.   
\end{remark}

Before moving on to the proof of Theorem \ref{thm:raney}, we need an important result called the \textit{cycle lemma} as stated by Dershowitz and Zaks \cite{DZ}, but originally due to Dvoretzky and Motzkin \cite{DM}.
\begin{lemma}\cite{DM, DZ}\label{lem:cyclic}
For any sequence $p_1, p_2, \ldots, p_{a+b}$ of $a$ boxes and $b$ circles, with $a\geq kb$, there exists exactly $a-kb$ (out of $a+b$) cyclic permutations $p_jp_{j+1}\cdots p_{a+b}p_1\cdots p_{j-1}$ with $1\leq j\leq a+b$ that are $k$-dominating. (A sequence $p_1, p_2, \ldots, p_{\ell}$ of boxes and circles is called $k$-dominating if for every position $i$, the number of boxes in $p_1p_2\cdots p_i$ is more than $k$ times the number of circles, with $1\leq i\leq \ell$.)
\end{lemma}

\begin{proof}[Proof of Theorem~\ref{thm:raney}]
The height of the path at each block $B_i$ is denoted by $h_i$ for $1\leq i\leq m$. By Lemma~\ref{lem:structure}, balanced paths are in bijection with their block
heights $0\le h_1\le\cdots\le h_m$. Since the heights are constant on blocks
and nondecreasing, the inequality $p_i\le i-1$ is tightest at the first
column of each block, giving us
\[
h_t\;\le\;\col(t,1)-1\;=\;c(t-1)+s,\qquad t=1,\dots,m.
\]
Such sequences are in bijection with lattice paths from $(0,0)$ to $(m,\,cm+s)$ using $E$ and $N$ steps 
staying weakly below $y=cx+s$: Starting at $(0,0)$ we insert $h_{t}-h_{t-1}$ $N$ steps before the $t$-th $E$ step and then pad with $N$ steps after the last $E$ step.

We now want to count these paths using Lemma \ref{lem:cyclic}. Append one
extra $N$-step to each path and let $v$ be the \emph{reversal} of the
resulting $E$-$N$ word formed by the lattice path. Clearly, $v$ has $cm+s+1$ letters $N$ and $m$ letters $E$. We
regard the $N$-steps as boxes and the $E$-steps as circles, and apply
Lemma~\ref{lem:cyclic} with $k=c$.

We claim that a word $v$ is $c$-dominating if and only if it arises from a
valid path. Indeed, suppose a prefix of the original word $u$ (path plus
appended $N$) contains $x$ $E$-steps and $y$ $N$-steps, then the complement of the
suffix, which is a prefix of $v$ of length $L-(x+y)$ where $L=(c+1)m+s+1=(c+1)m+r$,
contains $m-x$ $E$-steps and $(cm+s+1)-y$ $N$-steps, and the $c$-dominance
condition for it is now
\[
(cm+s+1)-y \;>\; c\,(m-x),
\]
which after cancelling $cm$ becomes $y<cx+s+1$, that is $y\le cx+s$ since all
quantities are integers.

As the prefix of $u$ runs over all proper prefixes, its complementary suffix
runs over all nonempty prefixes of $v$. Hence $v$ is $c$-dominating if and
only if $y\le cx+s$ holds at every prefix of the original path. Moreover
every $c$-dominating word begins with a box (since a word beginning with a circle
violates dominance at its first position), that is, with the appended
$N$-step, so removing that $N$ and reversing recovers a path. Thus, the
correspondence between valid paths and $c$-dominating words is a bijection.

We now count the $c$-dominating words by a double counting argument. Consider the pairs
$(w,j)$, where $w$ is a word consisting of $cm+s+1$ boxes and $m$ circles,
$1\le j\le L$, and the $j$-th cyclic shift of $w$ is $c$-dominating. On the
one hand, by Lemma~\ref{lem:cyclic} (applied with $L$ letters, $cm+s+1$
boxes, $m$ circles and $k=c$), every word $w$ contributes exactly
\[
(cm+s+1)-cm\;=\;s+1\;=\;r
\]
such pairs, for a total of $r\dbinom{L}{m}$. On the other hand, every
$c$-dominating word $v$ arises from exactly $L$ pairs, since for each $j$
there is a unique word whose $j$-th cyclic shift equals $v$. Hence the number
of $c$-dominating words is
\[
\frac{r}{L}\binom{L}{m}
=\frac{r}{(c+1)m+r}\binom{(c+1)m+r}{m}
=\Raney_{c+1,r}(m),
\]
which completes the proof.
\end{proof}

\end{document}